\documentclass[12pt]{amsart}

\usepackage{amssymb}
\usepackage{amsthm}
\usepackage{amsmath}
\usepackage{amscd}
\usepackage{verbatim}
\usepackage[all]{xy}

\numberwithin{equation}{section}

\theoremstyle{plain}
\newtheorem{theorem}{Theorem}[section]
\newtheorem{corollary}[theorem]{Corollary}
\newtheorem{lemma}[theorem]{Lemma}
\newtheorem{proposition}[theorem]{Proposition}

\theoremstyle{definition}
\newtheorem{definition}[theorem]{Definition}
\newtheorem{remark}[theorem]{Remark}

\theoremstyle{remark}

\newcommand{\OO}{\mathcal O}
\newcommand{\A}{\mathbb{A}}
\newcommand{\R}{\mathbb{R}}

\newcommand{\Q}{\mathbb{Q}}
\newcommand{\Z}{\mathbb{Z}}

\newcommand{\C}{\mathbb{C}}

\renewcommand{\H}{\mathbb{H}}

\newcommand{\D}{\mathbb{D}}

\newcommand{\kzxz}[4]{\left(\begin{smallmatrix} #1 & #2 \\ #3 & #4\end{smallmatrix}\right) }

\newcommand{\tr}{\operatorname{tr}}

\newcommand{\Spin}{\operatorname{Spin}}

\newcommand{\Mp}{\operatorname{Mp}}

\newcommand{\End}{\operatorname{End}}
\newcommand{\Iso}{\operatorname{Iso}}

\newcommand{\Pic}{\operatorname{Pic}}
\newcommand{\Div}{\operatorname{Div}}

\newcommand{\diag}{\operatorname{diag}}
\newcommand{\ord}{\operatorname{ord}}
\newcommand{\kay}{k}

\newcommand{\ff}{\hbox{if }}
\newcommand{\SL}{\operatorname{SL}}

\newcommand{\CH}{\operatorname{CH}}
\newcommand{\GS}{\operatorname{GS}}

\newcommand{\MW}{\operatorname{MW}}

\begin{document}

\title{Twisted arithmetic Siegel Weil formula on $X_{0}(N)$}

\date{\today}
\author[Tuoping Du]{Tuoping Du}
\address{Department of Mathematics, Northwest University, Xi'an, 710127 ,  P.R. China}
\email{dtp1982@163.com}
\author[Tonghai Yang]{Tonghai Yang}
\address{Department of mathematics, Wisconsin University}
\email{thyang@math.wisc.edu}

\begin{abstract} In this paper,  we study twisted  arithmetic divisors on the modular curve $\mathcal X_0(N)$ with $N$ square-free. For each pair $(\Delta, r)$ where $\Delta >0$ and $\Delta \equiv r^2 \mod 4N$, we constructed a twisted arithmetic theta function $\widehat{\phi}_{\Delta, r}(\tau)$ which is a generating function of arithmetic twisted Heegner divisors. We prove the modularity of $\widehat{\phi}_{\Delta, r}(\tau)$, along  the way, we also identify the arithmetic pairing  $\langle  \widehat{\phi}_{\Delta, r}(\tau),\widehat{\omega}_N \rangle$  with special value of some Eisenstein series, where $\widehat{\omega}_N$ is a normalized  metric Hodge line bundle.

\end{abstract}

\dedicatory{}


\thanks{The first author is  partially supported by NSFC(No.11401470), and the second author is partially supported by a NSF grant DMS-1500743.}


\maketitle

\section{introduction}\label{introduction}\pagenumbering{arabic}\setcounter{page}{1}

 Eisenstein series plays  an important role in the arithmetic geometry and number theory. Kudla
conjectured that  the derivative of Eisenstein series is closely related to the arithmetic intersection number on Shimura varieties, commonly called  arithmetic Siegel-Weil formula.
In  \cite{KRYComp}, Kudla, Rapoport and Yang
  gave such a formula on a division Shimura curve.
 Kudla and Yang worked out a result  on the modular curve $X_{0}(1)$ \cite{Yazagier}. The associated Eisenstein series is   Zagier's famous Eisenstein series \cite{HZ} of weight $3/2$. In \cite{BF1}, Bruinier and Funke gave another proof of the main formula in \cite{Yazagier}  via theta lifting.  We extended  the arithmetic Siegel-Weil formula to  modular curve $X_{0}(N)$ with $N$ square free in \cite{DY}.

 This paper is a sequel to \cite{DY}.  We replace the Heegner divisors by twisted Heegner divisors, which  were first studied  by   Bruinier and Ono in \cite{BruOno}. Interestingly, the derivative part of the Eisenstein series disappear in such a case.

  Let $N >0$ be a positive integer, and
let   \begin{equation}
V=\{w=\kzxz
    {w_1} { w_2 }
      {w_3}   {-w_1}  \in M_{2}(\Q) : \tr(w) =0 \},
\end{equation}
with the quadratic form $Q(w)=N \det{w}=-Nw_2w_3-Nw_1^{2}$.  Let
 \begin{equation}
L=\big\{w =\kzxz {b}{\frac{-a}{N}}{c}{-b}
   \in M_{2}(\Z) \mid   a, b, c \in \Z \big\}\subset V,
\end{equation}
and let $L^\sharp$ be its dual lattice.
Then $\SL_2\cong \Spin(V)$ acts on $V$  by conjugation, i.e.,  $g.w=gwg^{-1}$, and $\Gamma_0(N)$ acts on $L^\sharp/L$ trivially. Let $\D$ be the  associated Hermitian domain  of positive lines in  $V_\R$, then it is isomorphic to upper half plane $\H$ via (\ref{inverse}) which preserves the action of $\SL_2$. We
identify $X_{0}(N)$ with compactification of $\Gamma_{0}(N)\setminus \D$.

 For each $\mu \in  L^\sharp/L$, let  $L_{\mu}=\mu+L$ and
$$
L_\mu[n] =\{ w \in  L_\mu:\,  Q(w) = n\}.
$$

Let $\Delta \in \Z_{>0}$ be a positive fundamental discriminant which is a square modulo $4N$. Let $L^\Delta=\Delta L$ with quadratic form $Q^\Delta(x)= \frac{Q(x)}{\Delta}$, then its dual lattice $L^{\Delta,\sharp} =L^\sharp$.  Associated to $\Delta$ is a generalized `genus character'
$\chi_\Delta: L^\sharp/L^\Delta \rightarrow \{ \pm 1\}$, defined by Gross, Kohnen, and Zagier in \cite{GKZ}, which can be rephrased as a map (see Section \ref{sect:TwistedDivisor} for detail).  Fix a class $r \mod (2N)$ with $\Delta =r^2 \mod (4N)$, we have then a twisted Heegner divisor for
 any $\mu \in L^{\sharp}/L$ and a positive rational number $n \in Q(\mu) +\Z$,
 \begin{equation}\label{TwistedHeegner}
Z_{\Delta, r}(n, \mu):=\sum_{w \in \Gamma_{0}(N)\setminus L_{r\mu}[ \Delta n]}\chi_{\Delta}(w)Z(w) \in \Div(X_{0}(N))_{\Q},
\end{equation}
which is defined over $\Q(\sqrt{\Delta})$. Here $Z(w) = \R w$ is the point on $X_0(N)$ given by the positive line $\R w$.

We will construct  twisted Kudla's Green function $\Xi_{\Delta, r}(n , \mu, v)$  for $Z_{\Delta, r}(n, \mu)$  in Section \ref{sect:KudlaGreenFunction}. All these functions are smooth at cusps, which are different from the Green functions $\Xi(n, \mu, v)$ in \cite{DY}. These functions are well-defined and smooth when  $n \le 0$.

Now assume that $N$ is square free. Let $\mathcal X_0(N)$ be the  canonical integral model over $\Z$ of $X_0(N)$ as defined in \cite{KM} (see Section  \ref{sect:ArithIntersectionReview}).
 Then we could define twisted arithmetic divisors in arithmetic Chow group  $\widehat{\CH}^1_\R(\mathcal X_0(N))$ by \begin{eqnarray}
\widehat{\mathcal{Z}}_{\Delta, r}(n, \mu, v) = \begin{cases}
(\mathcal{Z}_{\Delta, r}(n, \mu), \Xi_{\Delta, r}(n , \mu, v)) &\ff n >0,
  \\
    (0, \Xi_{\Delta, r}(n , \mu, v)) &\ff \hbox{ otherwise},
    \end{cases}
\end{eqnarray}
where $\mathcal{Z}_{\Delta, r}(n, \mu)$ is the Zariski closure of $Z_{\Delta, r}(n, \mu)$ in  $\mathcal X_0(N)$.
Now we could define the twisted arithmetic theta function as follows.
\begin{definition}
\begin{equation} \label{eq:twistedtheta}
\widehat{\phi}_{\Delta, r}(\tau) = \sum_{\substack{ n \equiv Q(\mu) (\mod \Z)\\
                          \mu \in L^\sharp/L
                          }}
                           \widehat{\mathcal Z}_{\Delta, r}(n, \mu, v) q_\tau^n e_\mu  \in \widehat{\CH}^1_\R(\mathcal X_0(N))\otimes \C[L^\sharp/L][[q, q^{-1}]],
\end{equation}
where $q_{\tau}=e^{2\pi i\tau}$.
\end{definition}
 Let $\Gamma'$ be the metaplectic cover of $\SL_2(\Z)$ which  acts on $\C[L^\sharp/L]$ via the Weil representation $\rho_L$ (see (\ref{eq:WeilRepresentation})) and let $\{ e_\mu:\, \mu \in  L^\sharp/L\}$ be the standard basis of $\C[L^\sharp/L]$. Then
$$
E_L(\tau, s) = \sum \limits_{ \gamma ^{\prime} \in \Gamma_{\infty} ^{\prime}\diagdown
\Gamma^{\prime}} \big(v^{\frac{s-1}{2}}e_{\mu_{0}} \big)\mid_{3/2, \rho_{L}} \gamma ^{\prime}
$$
is a vector valued Eisenstein series of weight $3/2$, where  the Petersson slash operator is defined on functions $f:\H\rightarrow \C[L^{\sharp} / L]$ by
 $$\big(f \mid_{3/2, \rho_{L}} \gamma ^{\prime}\big)(\tau)=\phi(\tau)^{-3}\rho_{L}^{-1}(\gamma ^{\prime})f(\gamma \tau),$$ and $\gamma ^{\prime}=(\gamma,
 \phi)\in \Gamma^{\prime} $.  Define  the normalized Eisenstein series \cite[equation (1.5)]{DY}
\begin{equation}\label{vectormodularform}
\mathcal{E}_{L}(\tau,s)=- \frac{s}{4} \pi^{-s-1}\Gamma(s)\zeta^{(N)}(2s)N^{\frac{1}{2}+\frac{3}{2}s} E_L(\tau, s),
\end{equation}
where
$$
\zeta^{(N)}(s) = \zeta(s) \prod_{p|N} (1-p^{-s}).
$$

Let $\widehat{\omega}_N$ be the  metrized Hodge bundle on $\mathcal X_0(N)$ with certain  normalized Petersson metric (see (\cite{Kuhn2}) and (\ref{eq1.11})).
The main result of this paper is as follows.
\begin{theorem} \label{Main} Let  $\Delta>1$ be a fundamental discriminant, then
\begin{equation}\label{degree}
\deg \widehat{\phi}_{\Delta, r}(\tau)= 0
\end{equation}
and
\begin{equation}\label{intersectionformula}
\langle \widehat{\phi}_{\Delta, r}(\tau), \widehat{\omega}_N \rangle = \frac{1}{\varphi(N)}\log(u_{\Delta})h(\Delta)\mathcal{E}_{L}(\tau, 1),
\end{equation} where $u_{\Delta}>1$ is the fundamental unit of quadratic field  $\Q(\sqrt{\Delta})$, $h(\Delta)$ is its  class number and $\varphi$ is the Euler function.
\end{theorem}

It is interesting to compare it with the main result in  \cite{DY}, which we state it here for convenience.

\begin{theorem}\cite[Theorem 1.3]{DY}  When $\Delta =1$,
$$
\deg \widehat{\phi}_{\Delta, r}(\tau)= \frac{2}{\varphi(N)}\mathcal E_L(\tau, 1),
$$
and
$$
\langle \widehat{\phi}_{\Delta, r}(\tau), \widehat{\omega}_N \rangle =\frac{1}{\varphi(N)}\bigg(\mathcal E_L'(\tau, 1)-\sum_{p|N} \frac{p}{p-1} \mathcal E_L(\tau, 1) \log p \bigg).
$$
\end{theorem}

Here is the basic idea in  the proof of Theorem \ref{Main}.
 For any $\Gamma_{0}(N)$ invariant function $f$, we define the twisted theta lift  by
\begin{equation}
 I_{\Delta, r}(\tau, f)=\int_{X_{0}(N)}f(z)\Theta_{\Delta, r}(\tau, z),
 \end{equation}
where $ \Theta_{\Delta, r}(\tau, z)$ is the twisted Kudla-Millson theta kernel defined by (\ref{eq:TwistedTheta}), following  \cite[Section 4]{AE}.

Define the normalized Eisenstein series of weight $0$ as in  \cite{DY} by
$$\mathcal{E}(N, z, s)=N^{2s}\pi^{-s}\Gamma(s)\zeta^{(N)}(2s)\sum_{\gamma \in \Gamma_{\infty }\setminus\Gamma_0(N)}(\Im(\gamma z))^{s}.$$
 We prove the following theorem in Section \ref{sectheta}.

\begin{theorem}\label{eisensteinlift}
 When $\Delta$ is a fundamental discriminant,
 \begin{equation}\label{lifteisenstein}
 I_{\Delta, r}(\tau, \mathcal{E}(N, z, s))=\Delta^{\frac{s}{2}}\Lambda(\varepsilon_{\Delta},
 s)\mathcal{E}_{L}(\tau, s),
 \end{equation}
  where $\Lambda(\varepsilon_{\Delta}, s)=L(\varepsilon_{\Delta},s)\Gamma(\frac{s}{2})\pi^{-\frac{s}{2}}$ is the completed L-series associated to the character $\varepsilon_{\Delta}(n)=\big(\frac{\Delta}{n}\big)$.
 \end{theorem}
Take the residue of both sides, we obtain the first identity in Theorem \ref{Main}.
For a modular form $f$ of weight $k$, we define its  renormalized Petersson metric as
\begin{equation} \label{eq1.11}
\| f(z)\| = |f(z) (4 \pi e^{-C} y)^{\frac{k}2}|
\end{equation}
where $C=\frac{\log 4\pi +\gamma}{2}$.

Combining the above theorem with the  Kronecker limit formula for $\Gamma_{0}(N)$ \cite[Theorem 1.5]{DY}, we obtain  the following result.
\begin{theorem}\label{theoremlift}
\begin{equation}
-\frac{1}{12}\emph{I}_{\Delta, r}(\tau, \log \parallel \Delta_{N} \parallel)= \begin{cases} \mathcal{E}_{L}^{\prime}(\tau,1) &\ff \Delta=1,\\
\log(u_{\Delta})h(\Delta)\mathcal{E}_{L}(\tau, 1)& \ff \Delta>1.
\end{cases}
\end{equation}
\end{theorem}
From above two theorems, we can reduce the proof of the second identity in   Theorem \ref{Main} to the  comparison between  the Fourier coefficients of $\emph{I}_{\Delta, r}(\tau, \log \parallel \Delta_{N} \parallel)$ and intersection number $< \widehat{\mathcal Z}_{\Delta, r}(n, \mu, v),\widehat{\omega}_N  >$. We will do  it in the Section \ref{secmainresult}.

Finally, we obtain the following modularity result.
\begin{theorem}\label{modular}
$\widehat{\phi}_{\Delta, r}(\tau)$ is a vector valued modular form for $\Gamma'$ of weight $\frac{3}2$, valued in $\C[L^\sharp/L]\otimes \widehat{\CH}^1_\R(\mathcal X_0(N))$.
\end{theorem}
The case  $\Delta=1$ is proved  in \cite{DY}.
In  \cite[Section 6]{BruOno}, Bruinier and Ono proved that
 $$A_{\Delta, r}(\tau)=\sum_{\mu \in L^{\sharp}/ L}\sum_{n>0}Z_{\Delta, r}(n, \mu)q_{\tau}^ne_{\mu}$$ is a cusp  form valued in $S_{\frac{3}{2}, \rho_L}\bigotimes J(\Q(\sqrt{\Delta}))$, where $J$ is the Jacobian of $X_{0}(N)$.
Notice  $A_{\Delta, r}(\tau)$ is the generic component of $\widehat{\phi}_{\Delta, r}(\tau)$. So modularity of $\widehat{\phi}_{\Delta, r}(\tau)$ (Theorem \ref{modular}) is an integral version of their result.

This paper is organized as follows.  We will introduce some notations and introduce the twisted Kudla-Millson theta function in  Section \ref{preliminaries}.  We will introduce the twisted theta lifting and use it to prove the Theorem \ref{eisensteinlift} and \ref{theoremlift} in Section \ref{sectheta}. In  Section \ref{sect:KudlaGreenFunction}, we  define twisted Kudla's Green functions, and show these functions are smooth at cusps. In  Section \ref{secmainresult}, we will define the arithmetic theta functions and prove the main result Theorem \ref{Main}. In the last section, we obtain the modularity Theorem \ref{modular}.

\textbf{Acknowledgments.} Add later.

\part{Theta lifting  and Kronecker limit formula}

\section{Basic set-up and theta lifting } \label{preliminaries}

 Let $V$ be the quadratic space and let $L$ be the even integral lattice defined in the introduction.
 Then $\SL_2\cong \Spin(V)$ acts on $V$  by conjugation, i.e.,  $g.w=gwg^{-1}$. Notice that $\Gamma_0(N)$  acts on $L^\sharp/L$ trivially.
Let  $\D$ be the Hermitian domain of  positive real lines in $V_\R$:
$$
\D=\{z \subset V_{\R}; \dim z=1 ~and~(~,~) \mid_{z} >0\}.
$$
The isomorphism between  $\H$ and  $\mathbb D$ is given by  the
map
\begin{equation}  \label{eq:wz}
z=x+i y  \mapsto  w(z) = \frac{1}{\sqrt{N} y} \left(
  \begin{array}{cc}
  -x  & z\overline{z}\\
     -1&x\\
  \end{array}
\right).
\end{equation}
The inverse  is
\begin{equation}\label{inverse}
w=\kzxz {a} {b} {c} {-a} \mapsto z(w) =\frac{2aN + \sqrt D}{2cN},  \quad D= - 4N Q(w).
\end{equation}
This isomorphism is $\SL_2(\R)$-action compatible and induces an
isomorphism between
$Y_0(N) =\Gamma_0(N) \backslash \H$ and $\Gamma_0(N) \backslash \mathbb D$.
Let  $X_0(N)$ be the usual compactification of $Y_0(N)$.

For $w \in V(\Q)$ with $Q(w)>0$,  the Heegner points $Z(w)$ is the image of $z(w)$ in $X_0(N)$; when $Q(w)\leq 0$, set $z(w)=0$.

\subsection{The Weil representation}

Let $\Mp_{2, \R}$ be the metaplectic double cover of $\SL_2(\R)$, which can be  realized  as pairs $(g, \phi(g, \tau))$, where $g=\kzxz{a}{b}{c}{d} \in \SL_2(\R)$, $\phi(g,
\tau)$ is a holomorphic function of $\tau \in \H$ such that $\phi(g, \tau)^2 = j(g, \tau) = c\tau +d$. Let $\Gamma'$ be the preimage of
$\SL_2(\Z)$ in $\Mp_{2, \R}$ with two generators
$$
S=\left( \kzxz {0} {-1} {1} {0}, \sqrt \tau \right)  \quad  T= \left( \kzxz {1} {1} {0} {1} , 1 \right).
$$
We denote the standard basis of $\C[L^\sharp/L]$ by $\{ e_\mu=L_\mu:\, \mu \in L^\sharp/L\}$.  Then there is a Weil representation $\rho_L$ of $\Gamma'$ on $\C[L^\sharp/L]$ given by
\cite{Bo}
\begin{eqnarray} \label{eq:WeilRepresentation}
&\rho_{L} (T)e_{\mu}&=e(Q(\mu))e_{\mu} ,\\
&\rho_{L}(S) e_{\mu}&= \frac{e(\frac{1}8)}{\sqrt{\vert L^{\sharp} / L \vert}}\sum \limits_{\mu^{\prime} \in L^{\sharp} / L }e(-(\mu,
\mu^{\prime}))e_{\mu^{\prime}}\nonumber .
\end{eqnarray}
This Weil representation $\rho_L$ is closed related to the Weil representation
$\omega$ of $\Mp_{2, \A}$ on $S(V_\A)$( see \cite{BHY}).

The slash operator with weight $\frac{3}{2}$ is defined on functions $f: \H\rightarrow \C[L^{\sharp}/L]$ given by \begin{equation}(f\mid _{\frac{3}{2}, \rho_{L}} \gamma^{\prime})(\tau)= \phi(\tau)^{-3}\rho_{L}^{-1}(\gamma^{\prime})f(\gamma \tau),\end{equation}
where $\gamma^{\prime}=(\gamma, \phi) \in \Gamma^{\prime}.$

\subsection{Twisted Heegner divisors} \label{sect:TwistedDivisor}

 Let $\Delta \in \Z_{>0}$ be a fundamental  discriminant which is a square modulo $4N$, and let $L^\Delta= \Delta L$ with renormalized quadratic form $Q_\Delta(w) = \frac{Q(w)}{\Delta}$. Then
it is easy to check $L^{\Delta, \sharp} = L^\sharp$. Let $\Gamma_\Delta$ be the subgroup of $\Gamma_0(N)$ which acts on $L^{\Delta, \sharp}/L^{\Delta}$ trivially. It is not hard to check that the map
\begin{equation}
\chi_{\Delta}(\kzxz{\frac{b}{2N} }{\frac{-a}{N}}{c}{-\frac{b}{2N}}) =\begin{cases}
  (\frac{\Delta}{n}),  &\ff \Delta \mid b^2-4Nac~and~ \frac{b^2-4Nac}{\Delta}~is~ a\\
  &  ~square ~modulo~4N~and~(a, b, c, \Delta)=1,
  \\
  0,  &otherwise.
 \end{cases}
\end{equation}
gives  a well-defined map
$$
\chi_\Delta:  L^{\Delta, \sharp}/L^\Delta  \rightarrow \{ \pm 1 \}.
$$
Here $n$ is any integer prime to $\Delta$ represented by one of the quadratic form $[N_1a, b, N_2c]$ with $N_1N_2=N$ and $N_1$, $N_2>0$, and $[a, b, Nc]=ax^2 + b xy + Nc y^2$ is the integral binary quadratic form corresponding to $w=\kzxz{\frac{b}{2N} }{\frac{-a}{N}}{c}{-\frac{b}{2N}}$.  Indeed,  $\chi_\Delta(w) = \chi_\Delta ([a, b, Nc])$ is the generalized genus character defined in \cite[Section 1]{GKZ} (see also \cite[Section 4]{BruOno}). We leave it to readers to check that $\chi(w +L^\Delta) = \chi(w)$ and so $\chi_\Delta$ induces a map on $L^{\Delta, \sharp}/L^\Delta $.
 It is known \cite{GKZ} that  the map is invariant under the action of $\Gamma_{0}(N)$ and the action of all Atkin-Lehner involutions, i,e.,
\begin{equation}
\chi_{\Delta}(\gamma w\gamma^{-1})=\chi_{\Delta}(w)~ and ~\chi_{\Delta}(W_{M}wW_{M}^{-1})=\chi_{\Delta}(w),
\end{equation}
where $\gamma \in \Gamma_{0}(N)$ and $W_{M}$ is the Atkin-Lehner involution with $M \mid N$.

Fix a $r \mod 2N$ class with $r^2 \equiv \Delta \mod 4N$.
For any $\mu\in L^{\sharp}/L$ and a positive rational number $n \in Q(\mu) +\Z$, we define the twisted Heegner divisor  by
\begin{equation}\label{TwistedHeegner}
Z_{\Delta, r}(n, \mu):=\sum_{w \in \Gamma_0(N) \setminus L_{r\mu}[n \Delta ]}\chi_{\Delta}(w)Z(w) \in \Div(X_{0}(N))_{\Q},
\end{equation}
which is defined over $\Q(\sqrt{\Delta})$. Notice that we count each  point $Z(w)=\R w$  with  multiplicity $\frac{2}{|{\Gamma}_w|}$ in the orbifold $X_0(N)$, where $\Gamma_w$ is the stabilizer of $w$ in $\Gamma_0(N)$. So our definition is the same as that in  \cite[Section 5]{AE} and \cite[Section 5]{BruOno}.

Now define
\begin{equation}
Z_{\Delta}(n, \delta):=\sum_{w \in \Gamma_{\Delta} \setminus L^{\Delta}_\delta[ n]}Z(w) \in \Div(X_{\Gamma_{\Delta}}).
\end{equation}

Recall the natural map $$\pi_{\Gamma_{\Delta}}: X_{\Gamma_{\Delta}}\longrightarrow X_{0}(N)$$
is a covering map with the degree
$[\overline{\Gamma_{0}(N)}: \overline{\Gamma_{\Delta}}]$. Here for any congruence group, $\bar{\Gamma}=\Gamma/(\Gamma\cap \{\pm 1\})$.



\begin{lemma}\label{pullbackdivisor}
Let  $n\equiv Q(\mu) (\mod \Z)$ be a positive number, then
\begin{equation}
\sum_{\delta \in L^{\Delta,  \sharp}/L^{\Delta}, \delta\equiv r \mu( L)  }\chi_{\Delta}(\delta) Z_{\Delta}(n, \delta)= \pi_{\Gamma_{\Delta}}^{\ast}(Z_{\Delta, r}(n, \mu)),
\end{equation}
where $\pi_{\Gamma_{\Delta}}^{\ast}$ is the pullback
$$\pi_{\Gamma_{\Delta}}^{\ast} : Z^1(X_{0}(N)) \longrightarrow Z^1(X_{\Gamma_{\Delta}}) .$$
\end{lemma}
\begin{proof} Write $\Gamma =\Gamma_0(N)$, and notice $L^\sharp =L^{\Delta, \sharp}$.
 For $w \in \Gamma \setminus L_{r\mu}[\Delta n]$, $\pi_{\Gamma_{\Delta}}^{-1}(Z(w))=\{Z(w_{1}), ... ,Z(w_{g})\}$, then $\pi_{\Gamma_{\Delta}}^{\ast}(Z(w))=Z(w_{1})+...+Z(w_{g})$.

Thus one has
\begin{eqnarray}
\pi_{\Gamma_{\Delta}}^{\ast}(Z_{\Delta, r}(n, \mu))
&=&\sum_{w \in \Gamma \setminus L_{r\mu}[ \Delta n]}\chi_{\Delta}(w)\pi_{\Gamma_{\Delta}}^{\ast}(Z(w))\nonumber\\
&=&\sum_{\delta \in L^{\sharp}/L^{\Delta}, \delta\equiv r \mu (L) }\chi_{\Delta}(\delta)\sum_{w \in \Gamma_{\Delta} \setminus L^{\Delta}_{\delta}[ n]}Z(w)\nonumber\\
&=&\sum_{\delta \in L^{\sharp}/L^{\Delta}, \delta\equiv r \mu(L)  }\chi_{\Delta}(\delta) Z_{\Delta}(n, \delta)
\end{eqnarray}

Then we obtain the result.

\end{proof}

\subsection{Twisted Kudla-Millson theta function}
Following Kudla and Millson  (\cite{KMi}, \cite[Section 3]{BF1}),
 for  $z=x+iy \in \H$, there is a decomposition
$$
V_\R =\R w(z) \oplus \R w(z)^\perp, \quad w = w_z + w_{z^\perp}.
$$
Define
$R(w, z)_{\Delta} =-(w_{z^\perp}, w_{z^\perp})_{\Delta}$, and the majorant
$$
(w, w)_{\Delta, z}= (w_z, w_z)_{\Delta} + R(w, z)_{\Delta},
$$
where  $(,)_{\Delta}=\frac{(~,~)}{ \Delta}$ is the bilinear form associated to the quadratic form $Q_\Delta$.
One has
\begin{align} \label{eq3.5}
R(w, z)_{\Delta}&=\frac{1}{2}(w, w(z))_{\Delta}^{2}-(w, w)_{\Delta}.
\end{align}
For  $w =\kzxz
    {w_{1}}{w_{2}}
     {w_{3}}{-w_{1}}\in V_{\R},$
\begin{equation} \label{formula1}
(w, w(z))_{\Delta}=-\frac{\sqrt{N}}{y\sqrt{ \Delta}}(w_3z\overline{z}-w_1(z+\overline{z})-w_2).
\end{equation}

Let
\begin{align} \label{eq3.5}
\varphi^{0}_{\Delta}(w, z)&=\bigg((w, w(z))_{\Delta}^{2}- \frac{1}{2\pi}\bigg)e^{-2\pi R(w, z)_{\Delta}}\mu(z) \notag
\\
\quad \hbox{ and } \varphi_{\Delta}(w, \tau, z) &= e(Q_{\Delta}(w)\tau) \varphi^0_{\Delta}(\sqrt v w, z) ,
\end{align}
be the  differential forms on $V_{\R}$ valued in $\Omega^{1,1}(\D)$, where $\mu(z) =\frac{dx \, dy}{y^2}$.

For any $\delta \in L^{\Delta, \sharp}/L^\Delta$, define
\begin{align}
\Theta_\delta(\tau, z) &=\sum_{w \in  L^{\Delta}_\delta}  \varphi_{\Delta}( w, \tau, z) ,
\end{align}
where $L^{\Delta}_\delta=L^{\Delta}+\delta$.
Then
\begin{equation} \label{eq:ThetaCoefficient}
\Theta_{L^\Delta}(\tau, z) =   \sum_{\delta \in L^\sharp/L^{\Delta} }  \Theta_\delta(\tau, z) e_\delta
\end{equation}
is a vector valued Kudla-Millson theta function, which is a nonholomorphic modular form of weight $3/2$  of $(\Gamma',  \rho_{L^\Delta })$ with respect to the variable $\tau$ with values in $ \Omega^{1, 1}(X_{\Gamma_{ \Delta}}) $, where $X_{\Gamma_{\Delta}} $ is the modular curve $\Gamma_{\Delta} \backslash \H^{\ast} $.

Fix a class $r \mod 2N$ with $r^2 \equiv \Delta \mod 4N$. Following the Bruinier and Ono's work \cite{BruOno},  Alfes and Ehlen constructed a $\C[L^\sharp/L]-$ valued theta function \cite[Section 4]{AE}
\begin{equation} \label{eq:TwistedTheta}
\Theta_{\Delta, r}(\tau, z):=\sum_{\mu}\Theta_{\Delta, r, \mu}(\tau, z)e_{\mu},
\end{equation}
where $$\Theta_{\Delta, r, \mu}(\tau, z)=\sum_{\delta \in L^{\sharp}/L^{\Delta},\ \delta\equiv r \mu( L), \atop Q_{\Delta}(\delta)\equiv Q(\mu)(\mod 1)}\chi_{\Delta}(\delta)\Theta_\delta(\tau, z) .$$

This twisted theta function has a good transformation properties just like the classical Kudla-Millson theta functions.
\begin{proposition}\cite[Proposition 4.1]{AE}
The theta function $\Theta_{\Delta, r}(\tau, z)$ is a non-holomorphic $\C[L^\sharp/L]$-valued modular form of weight $3/2$ for the representation
$\rho_L$. Furthermore, it is a non-holomorphic automorphic form of weight $0$ for $\Gamma_{0}(N)$ in the variable $z \in \D$.
\end{proposition}

\section{Twisted theta lift}\label{sectheta}
 Following  Alfes and Ehlen \cite{AE}, we consider the twisted theta lifting:  for any $\Gamma_{0}(N)$-invariant function  $f(z)$,  the lifting is given by
 \begin{equation}
 I_{\Delta, r}(\tau, f):=\int_{X_{0}(N)}f(z)\Theta_{\Delta, r}(\tau, z)=\sum_{\mu \in L^\sharp/L}\int_{X_{0}(N)}f(z)\Theta_{\Delta, r, \mu}(\tau, z)e_{\mu},
 \end{equation}
if the integral is convergent.

In this section, we consider the lift of Eisenstein series $\mathcal{E}(N, z, s)$ and Petersson norm $\log \| \Delta_{N} \|$. From \cite [ Proposition 4.1]{BF1}, one knows the theta function is $O( e^{-Cy^{2}})$ around cusps, ~as~$y\longrightarrow \infty$  for some constant $ C> 0$. Then these two lifts are convergent.

We first recall a result of Alfes and Ehlen.

 \begin{proposition}(\cite[Proposition 3.1]{Al}\label{thetakernal}, \cite{Eh})
 Let $K=\Z$ with the quadratic form $Q(x)=-Nx^2$  ($K^\sharp/K \cong L^\sharp/L$), then
 \begin{eqnarray}
 &&\Theta_{\Delta, r}(\tau, z)=-y\frac{N^{3/2}}{2 \Delta }\sum_{n=1}^{\infty}\sum_{\gamma \in \Gamma_{\infty} ^{\prime}\diagdown\Gamma^{\prime}}n^{2}(\frac{\Delta}{n}) \nonumber\\
 &&\times\bigg[exp(-\pi\frac{y^{2}Nn^{2}}{v \Delta})v^{-3/2}\sum_{\lambda \in K^{\sharp}}e( \Delta Q(\lambda)\overline{\tau}-2N\lambda nx )e_{r\lambda} \bigg]\vert_{3/2, \rho_{K}}\gamma dxdy.\nonumber
 \end{eqnarray}

 \end{proposition}
 Now we are ready to prove   Theorem \ref{eisensteinlift}  which we restate here for convenience. We follow the idea in the proof of \cite[Theorem 6.1]{AE}, where the case $N=1$ is considered.
 \begin{theorem}\label{theo:eisensteinlift}
 When $\Delta$ is a fundamental discriminant ,
 \begin{equation}\label{lifteisenstein}
 I_{\Delta, r}(\tau, \mathcal{E}(N, z, s))=\Delta^{\frac{s}{2}}\Lambda(\varepsilon_{\Delta},
 s)\mathcal{E}_{L}(\tau, s),
 \end{equation}
 where $\Lambda(\varepsilon_{\Delta}, s)=L(\varepsilon_{\Delta},s)\Gamma(\frac{s}{2})\pi^{-\frac{s}{2}}$ is the completed L-series associated to the character $\varepsilon_{\Delta}$.
 \end{theorem}
 \begin{proof}  One has by  Proposition \ref{thetakernal},
 \begin{eqnarray}
 &&\Theta_{\Delta, r}(\tau, z)\nonumber\\
 &&=-y\frac{N^{3/2}}{ \Delta }\sum_{n=1}^{\infty}\sum_{\gamma \in \Gamma_{\infty} ^{\prime}\diagdown\Gamma^{\prime}}n^{2}\big(\frac{\Delta}{n}\big) \exp(-\pi\frac{y^{2}Nn^{2} \mid c\tau +d \mid^{2} }{v \Delta})v^{-3/2}\mid c\tau+d \mid^3\nonumber\\
 &&\times(c\tau+d)^{-3/2} \sum_{\lambda \in K^{\sharp}}e( \Delta Q(\lambda)\overline{ \tau}-2N\lambda nx )\rho_{K}^{-1}(\gamma)e_{r\lambda} dxdy\nonumber\\
&&=-\frac{yN^{3/2}}{ v^{3/2}\Delta }\sum_{n=1}^{\infty}\sum_{\gamma \in \Gamma_{\infty} ^{\prime}\diagdown\Gamma^{\prime}}n^{2}\big(\frac{\Delta}{n}\big) \exp(-\frac{\pi y^{2}Nn^{2} \mid c\tau +d \mid^{2} }{v \Delta})\nonumber\\
 &&\times(c\overline{\tau}+d)^{3/2} \sum_{\lambda \in K^{\sharp}}e(\Delta Q(\lambda)\overline{ \tau}-2N\lambda nx )\rho_{K}^{-1}(\gamma)e_{r\lambda} dxdy.\nonumber
 \end{eqnarray}
 Here for every coprime pair $(c,d)$, there is unique $\gamma=\kzxz{\ast}{\ast}{c}{d} \in \Gamma_{\infty} ^{\prime}\diagdown
\Gamma^{\prime}$, and we could identify them.

 Unfolding the integral, for $\Re(s)>1$, we have
\begin{eqnarray}\label{thetalift}
&&\emph{I}_{\Delta, r}(\tau, E(N, z, s))=\int_{\Gamma_{\infty}\diagdown\H} \Theta_{\Delta, r} (\tau, z)y^{s}\nonumber\\
&=&-\frac{N^{\frac{3}{2}}}{v^{\frac{3}{2}} \Delta }\sum_{n=1}^{\infty}n^{2}\big(\frac{\Delta}{n}\big)\sum \limits_{ \gamma \in \Gamma_{\infty} ^{\prime}\diagdown
\Gamma^{\prime}}(c\overline{\tau}+d)^{3/2}\int_{0}^{\infty}e^{\big(-\frac{\pi Ny^{2}n^{2}}{ \Delta v}\mid c\tau+d\mid^{2}\big)}y^{s+1}dy\nonumber
\\
&&\times \rho_{K}^{-1}(\gamma)\int_{0}^{1}\sum_{\lambda \in K^{\sharp}}e( \Delta Q(\lambda)\overline{ \tau}-2N\lambda nx )e_{r\lambda} dx.\nonumber
\end{eqnarray}
Notice that
$$
\int_{0}^{1}\sum_{\lambda \in K^{\sharp}}e( \Delta Q(\lambda)\overline{ \tau}-2N\lambda nx )e_{r\lambda} dx= e_{\mu_0}.
$$
Then one has

\begin{eqnarray}
&&\int_{\Gamma_{\infty}\diagdown\H} \Theta_L (\tau, z)y^{s}\nonumber
 \\
&=&-\frac{N^{\frac{3}{2}}}{2v^{\frac{3}{2}}\Delta }\sum_{n=1}^{\infty}n^{2}\big(\frac{\Delta}{n}\big)\sum \limits_{ \gamma \in \Gamma_{\infty} ^{\prime}\diagdown \Gamma^{\prime}}
\frac{v^{\frac{s+2}{2}} \Delta ^{\frac{s+2}{2}}(c\overline{\tau}+d)^{3/2}\Gamma\big(\frac{s}{2}+1\big)}{\pi^{\frac{s+2}{2}}\mid
c\tau+d\mid^{s+2}N^{\frac{s+2}{2}}n^{s+2}}\rho_{K}^{-1}(\gamma)e_{\mu_{0}} \nonumber\\
&=&-\frac{1}{2}N^{\frac{1-s}{2}} \Delta^{\frac{s}{2}} L(\varepsilon_{\Delta}, s)\Gamma\big(\frac{s}{2}+1\big)\sum \limits_{ \gamma  \in \Gamma_{\infty} ^{\prime}\diagdown
  \Gamma^{\prime}} \frac{v^{\frac{s-1}{2}}(c\overline{\tau}+d)^{3/2}}{\pi^{\frac{s}{2}+1}\mid
  c\tau+d\mid^{s+2}}\rho_{K}^{-1}(\gamma)e_{\mu_{0}} \nonumber\\
 &=&-N^{\frac{1-s}{2}}\Delta^{\frac{s}{2}} \frac{s}{4\pi}\Lambda(\varepsilon_{\Delta},s)\sum \limits_{ \gamma  \in \Gamma_{\infty} ^{\prime}\diagdown \Gamma^{\prime}}
 \big(v^{\frac{s-1}{2}}e_{\mu_{0}} \big)\mid_{3/2,\rho_{L}} \gamma. \nonumber
\end{eqnarray}
Notice in the last equality we identify $\rho_{L}$ with $\rho_{K}$ because the isomorphism  $K^{\sharp}/K\cong L^{\sharp}/L$  keep the quadratic form.
 For the  normalized Eisenstein series, we obtain
 \begin{eqnarray}
 \emph{I}_{\Delta, r}(\tau, \mathcal{E}(N, z, s))= \Delta^{\frac{s}{2}}\Lambda(\varepsilon_{\Delta},s)\mathcal{E}_{L}(\tau,s).\nonumber
 \end{eqnarray}
 \end{proof}

 Taking residue  of both sides of the equation (\ref{lifteisenstein}) at $ s=1$, we have the following result.
\begin{corollary}\label{thetaintegral}
 \begin{equation}
 \emph{I}_{\Delta, r}(\tau, 1)=\begin{cases} \frac{2}{\varphi(N)}\mathcal{E}_{L}(\tau,1) &\ff \Delta=1,\\
 0 &\ff \Delta >1.
 \end{cases}
 \end{equation}
\end{corollary}

Recall  the modular form of weight $k=12\varphi(N)$ for $\Gamma_0(N)$ defined in \cite[(1.6)]{DY}:
\begin{equation} \label{eq:DeltaN}
\Delta_N(z) =\prod_{t|N} \Delta(t z) ^{a(t)}
\end{equation}
with
$$
a(t) = \sum_{r|t} \mu(\frac{t}r) \mu(\frac{N}r) \frac{\varphi(N)}{\varphi(\frac{N}r)},
$$
where $\mu(n)$ is the the M\"obius function and $\varphi(N)$ is the Euler function.

For a modular form $f$ of weight $k$ and level $N$, we define its  normalized Petersson metric as
\begin{equation}
\| f(z)\| = |f(z) (4 \pi e^{-C} y)^{\frac{k}2}|
\end{equation}
where $C=\frac{\log 4\pi +\gamma}{2}$ with Euler constant $\gamma$.
\begin{theorem}\label{theo:theoremlift}
\begin{equation}
-\frac{1}{12}\emph{I}_{\Delta, r}(\tau, \log \| \Delta_{N} \|)= \begin{cases} \mathcal{E}_{L}^{\prime}(\tau,1) &\ff \Delta=1,\\
\log(u_{\Delta})h(\Delta)\mathcal{E}_{L}(\tau, 1)& \ff \Delta>1,
\end{cases}
\end{equation}
where $u_{\Delta} >1$ is the fundamental unit and $h(\Delta)$ is the class number of real quadratic field with discriminant $\Delta$.
\end{theorem}

\begin{proof}
We proved the case $\Delta=1$ in \cite[Theorem 1.6]{DY}. Now we assume $\Delta>1$.
From the Kronecker limit formula for $\Gamma_{0}(N)$\cite[Theorem 1.5]{DY}
$$
\lim _{s\rightarrow 1} \bigg( \mathcal{E}(N, z, s)-\varphi(N)\zeta^{\ast}(2s-1)\bigg)
=-\frac{1}{12}\log \big(y^{6\varphi(N)}\mid \Delta_{N}(z)\mid\big),
$$ one has
\begin{eqnarray}
&&-\frac{1}{12}\emph{I}_{\Delta, r}(\tau, \log  \mid\Delta_N(z) y^{6 \varphi(N)}\mid) \nonumber\\
&=&\lim _{s\rightarrow 1} \big(\emph{I}_{\Delta, r}(\tau, \mathcal{E}(N,z,s))
-\emph{I}_{\Delta, r}(\tau, \varphi(N)\zeta^{\ast}(2s-1))\big)\nonumber.
\end{eqnarray}
Here $\zeta^{\ast}(s)=\pi^{-\frac{s}{2}}\Gamma(\frac{s}{2})\zeta(s)$.
From the Theorem \ref{eisensteinlift} and Corollary \ref{thetaintegral}, we obtain
\begin{eqnarray}
-\frac{1}{12}\emph{I}_{\Delta, r}(\tau, \log  \|\Delta_N(z)\|) &=&\sqrt{\Delta}\Lambda(\varepsilon_{\Delta}, 1)\mathcal{E}_{L}(\tau, 1)\nonumber\\
&=&\log(u_{\Delta})h(\Delta)\mathcal{E}_{L}(\tau, 1).\nonumber
\end{eqnarray}
\end{proof}

\section{Twisted Kudla's  Green function} \label{sect:KudlaGreenFunction}

Following   Kudla's  methods in \cite{Kucentral}, we  construct twisted Kudla's Green function  for $Z_{\Delta, r}(n, \mu)$ in this section.

For $r>0$ and $s \in \R$, let
\begin{equation}\label{belta}
\beta_s(r) =\int_1^\infty e^{-rt } t^{-s} dt ,
\end{equation}
and
\begin{equation}
\xi_{\Delta}(w, z)=\beta_1(2\pi R(w, z)_ \Delta).
\end{equation}
\begin{definition}[Twisted Kudla's Green functions]

For $n \in Q(\mu)+\Z$, define
\begin{equation} \label{eq:TWGreen}
\Xi_{\Delta, r}(n , \mu, v)(z)=\sum_{\delta \in L^{\sharp}/L^{\Delta}, \delta\equiv r \mu( L), \atop Q_{\Delta}(\delta)=n}\chi_{\Delta}(\delta)
\Xi_{L^\Delta}\big(n, \delta, v\big)(z),
\end{equation}
where $\Xi_{L^\Delta} \big(n, \delta, v\big)(z)$ is the Kudla's Green function associated to the lattice $L^\Delta $ with quadratic form $Q_{\Delta}$ given by \begin{equation}
\Xi_{L^\Delta}\big(n, \delta, v\big)(z)=\sum_{ 0\neq w\in L_{\delta}^\Delta[n]}\xi_{\Delta}(\sqrt{v}w, z).
\end{equation}
\end{definition}
So one has
\begin{equation}
\Xi_{\Delta, r}(n , \mu, v)(z)=\sum_{0\ne w\in L_{r\mu}[ \Delta n]}\chi_{\Delta}(w)\xi_{\Delta}(\sqrt{v}w, z),
\end{equation}
which is clearly invariant under $\Gamma_0(N)$. Recall from the Lemma \ref{pullbackdivisor}
$$
\sum_{\delta \in L^{\sharp}/L^{\Delta}, \delta\equiv r \mu( L) }\chi_{\Delta}(\delta) Z_{\Delta}(n, \delta)= \pi_{\Gamma_{\Delta}}^{\ast}(Z_{\Delta, r}(n, \mu)).
$$
Since $\Xi_{L^\Delta}(n, \delta, v)$ is a Green function for $Z_{\Delta}(n, \delta)$ on $X_{\Gamma_\Delta}=\Gamma_\Delta \backslash \H$ by \cite{Kucentral}, we have thus the following lemma.

\begin{lemma}  When $ n>0$, $\Xi_{\Delta, r}(n , \mu, v)(z)$  is a Green function for $Z_{\Delta, r}(n, \mu)$ on $Y_{0}(N)=\Gamma_0(N) \backslash \H$, and satisfies the following current equation,
\begin{equation}d d^c  [\Xi_{\Delta, r}(n , \mu, v)(z)] +\delta_{Z_{\Delta, r}(n, \mu)}=[\omega_{\Delta, r}(n, \mu, v)],
\end{equation}
where  $\omega_{\Delta, r}(n, \mu,v)$ is the differential form
\begin{equation}
\omega_{\Delta, r}(n, \mu,v)=\sum_{w \in L_{\mu}[n \Delta]}\chi_{\Delta}(w)\varphi^{0}_{\Delta}(w, z).
\end{equation}
When $ n\le 0$, $\Xi_{\Delta, r}(n , \mu, v)(z)$ is smooth on $Y_0(N)$.
\end{lemma}
Now we consider behaviors of these Green functions  at cusps.   We recall some definitions  given by Funke and Bruinier, see \cite{BF} and \cite{Fu}. Let $\Iso(V)$ be the set of isotropic non-zero vectors $\ell=\Q w$ of $V$, with $0\ne w \in V$ and $Q(w)=0$. Given $\ell =span\{{\kzxz{a}{b}{c}{d}}\} \in \Iso(V)$, let $P_\ell= \frac{a}c$ be the associated cusp, which depends only on the equivalence  class of isotropic line $\ell$. Two isotropic lines give the same cusp in $Y_{0}(N)$ if and only if there is $\gamma \in \Gamma_{0}(N)$ such that $ \gamma \ell_1 = \ell_2$.

Let $\ell_\infty=span\{\kzxz {0} {1} {0} {0} \}\in \Iso(V)$ and then $P_\infty=\infty$ be the associated cusp. In general,
for an  isotropic line $\ell $, there exists $\sigma_{\ell} \in \SL_{2}(\Z)$ such that $\ell_\infty=  \sigma_{\ell}\ell.$
  Then
  $$\sigma_{\ell}\Gamma_{\ell}\sigma_{\ell}^{-1}=\{ \pm\kzxz {1}{m\kappa_{\ell}}{0}{1}, m\in \Z\},$$
   where $\Gamma_{\ell} \subseteq \Gamma_{0}(N)$ is the stabilizer of $\ell$ and $\kappa_{\ell} > 0$ is the classical width of the associated  cusp $P_{\ell}$.
   On the other hand, there is  another positive number $\beta_{\ell}> 0$, depending on $L$ and the cusp $P_\ell$,  such that $\kzxz{0}{\beta_{\ell}}{0}{0}$ is a primitive element in $\ell_\infty \bigcap \sigma_{\ell}L$.  The   Funke constant $\varepsilon_{\ell}=\frac{\kappa_{\ell}}{\beta_{\ell}}$  at cusp $P_\ell$ is defined  in \cite[Section 3]{Fu}, which is called width by Funke. We will simply denote $\kappa =\kappa_\infty$.

Let $ \delta_{w,\ell} $ be the number of  isotropic lines $ \ell_w \in \Iso(V)$  which is perpendicular to $w$ and belongs to the same cusp as $\ell$. We often drop the index $\ell$  when $\ell=\ell_{\infty}$.
In such a case, $w$ is orthogonal to two isotropic lines $ \ell_{w}=span\{X\}$ and $ \tilde{\ell}_{w}=span\{\tilde{X}\}$ such that $(w, X, \tilde{X})$ is a positively oriented basis of $V$. We denote $w \sim \ell_{w}$. Notice that $\tilde{\ell}_{w}=\ell_{-w}$.
 When $-4Nn$ is a square, Funke defined  in \cite[Section 3]{Fu}
 \begin{equation}
 L_{\mu, \ell}[n]=\{ w \in L_{\mu}[n]; w \sim \ell\},
 \end{equation}
 Here we use the different notation.

\begin{lemma}\label{lemmacoset}
Let $-4Nn$ be a square, then \begin{equation}
\sum_{w\in \Gamma_{0}(N)\setminus L_{\mu}[n]}\delta_{w}=\mid \Gamma_{\ell_{\infty}}\setminus L_{\mu, \ell_{\infty}}[n]\mid +\mid \Gamma_{\ell_{\infty}}\setminus L_{-\mu, \ell_{\infty}}[n]\mid.
\end{equation}
Moreover,
if $2\mu \notin L$ and  $ L_{\mu}[n] \neq \phi$ , then  one of the numbers $\mid\Gamma_{\ell_{\infty}}\setminus L_{\mu, \ell_{\infty}}[n]\mid, ~\mid\Gamma_{\ell_{\infty}}\setminus L_{-\mu, \ell_{\infty}}[n]\mid$ is $\sqrt{-4Nn}$, the other is $0$, so the sum is $\sqrt{-4Nn}$.
\end{lemma}
\begin{proof}
For any $w \in \Gamma_{0}(N) \setminus L_{\mu}[n]$ with $\delta_w \neq 0$, assume $w^\perp=\hbox{Span} \{ \ell_{\infty}, \ell\}$. From the definition of the $\delta_w$, this implies $ w \sim \ell_{\infty}$ or $-w \sim \ell_{\infty}$.

If  $\delta_{w}=1,$ then  $w \in L_{\mu, \ell_{\infty}}[n]$ or $- w \in L_{-\mu, \ell_{\infty}}[n]$.

If  $\delta_{w}=2$,    there exists $\gamma \in \Gamma_{0}(N)$ such that $\ell= \gamma \ell_{\infty}$. $w \in L_{\mu, \ell_{\infty}}[n]$($-w \in L_{-\mu, \ell_{\infty}}[n]$) implies that $-\gamma w \in L_{-\mu, \ell_{\infty}}[n]$($\gamma w \in L_{\mu, \ell_{\infty}}[n]$) respectively.

There is a map
\begin{equation}
f: \Gamma_{\ell_{\infty}}\setminus L_{\mu, \ell_{\infty}}[n]\bigcup \Gamma_{\ell_{\infty}}\setminus L_{-\mu, \ell_{\infty}}[n]\longrightarrow S\subseteq  \Gamma_{0}(N)\setminus L_{\mu}[n],
\end{equation}
given by if $w \in \Gamma_{\ell_{\infty}}\setminus L_{\mu, \ell_{\infty}}[n]$, $f(w)=w$; if $w \in \Gamma_{\ell_{\infty}}\setminus L_{-\mu, \ell_{\infty}}[n]$, $f(w)=-w$. Here $S$ is the subset $\{w \in \Gamma_{0}(N)\setminus L_{\mu}[n] \mid \delta_{w}\neq 0\}$.

Restriction on each subset $\Gamma_{\ell_{\infty}}\setminus L_{\mu, \ell_{\infty}}[n]$ or $ \Gamma_{\ell_{\infty}}\setminus L_{-\mu, \ell_{\infty}}[n]$, the map is injective. If $w, w^{\prime} \in \Gamma_{\ell_{\infty}}\setminus L_{\mu, \ell_{\infty}}[n]$ and $f(w)=f(w^{\prime})$, then there exists $\sigma \in \Gamma_{0}(N)$ such that $\sigma w= w^{\prime}$. It's need to keep orientation, so $\sigma \in \Gamma_{\ell_{\infty}}$, and then $w=w^{\prime} \in \Gamma_{\ell_{\infty}}\setminus L_{\mu, \ell_{\infty}}[n]$.

For $w \in S$, when $\delta_{w}=1$, there is only one preimage $w \in \Gamma_{\ell_{\infty}}\setminus L_{\mu, \ell_{\infty}}[n]$ or $- w \in \Gamma_{\ell_{\infty}}\setminus L_{-\mu, \ell_{\infty}}[n]$; when $\delta_{w}=2$, there are exactly two preimages $(w, -\gamma w)$ or $(\gamma w, - w) $ depends on $w \in L_{\mu, \ell_{\infty}}[n]$ or not.

From above discussions, counting both sides of $f$, one obtains
\begin{equation}
\sum_{w\in \Gamma_{0}(N)\setminus L_{\mu}[n]}\delta_{w}=\mid \Gamma_{\ell_{\infty}}\setminus L_{\mu, \ell_{\infty}}[n]\mid +\mid \Gamma_{\ell_{\infty}}\setminus L_{-\mu, \ell_{\infty}}[n]\mid.
\end{equation}
For any $\mu$, it's known that   $\mid \Gamma_{\ell_{\infty}}\setminus L_{\mu, \ell_{\infty}}[n]\mid=\sqrt{-4Nn}~ or ~0$(\cite[Section 3]{Fu}). From
 \cite[Lemma 6.2]{DY}, if $2\mu \notin L$
 $$\sum_{w\in \Gamma_{0}(N)\setminus L_{\mu}[n]}\delta_{w}=\sqrt{-4Nn}~ or ~0,$$
This proves the lemma.
\end{proof}
\begin{lemma}\label{charsum}
\begin{eqnarray}
 \sum_{w \in \Gamma_{\ell_{\infty}}\setminus L_{\mu, \ell_{\infty}}[n \Delta] }\chi_{\Delta}(w) =0.
\end{eqnarray}
\end{lemma}
\begin{proof}
The representatives for $\Gamma_{\ell_{\infty}}\setminus L_{\mu, \ell_{\infty}}[n \Delta]$
 is given by $$\bigg\{ \kzxz{k}{-j\beta_{\ell_{\infty}}}{}{-k}, j=0,..., 2kN-1\bigg\},$$ with $D=4N^2k^2$ and $\beta_{\ell_{\infty}}=\frac{1}{N}$.

 If $k$ is not of the form $\frac{\Delta k'}{2N}$ for some $k' \in \Z_{>0}$, then $\chi_{\Delta}(w)=0$ by  the definition of $\chi_{\Delta}$. So we could assume $k=\frac{\Delta k'}{2N}$. Then one has

 \begin{eqnarray}
 \sum_{w \in \Gamma_{\ell_{\infty}}\setminus L_{\mu, \ell_{\infty}}[n \Delta] }\chi_{\Delta}(w) = \sum_{j=0}^{2kN-1}\bigg(\frac{\Delta}{j}\bigg)
 =\sum_{j=0}^{\Delta k'-1}\bigg(\frac{\Delta}{j}\bigg)=0.
\end{eqnarray}
\end{proof}
The main purpose of this section is  to prove the following result.

\begin{theorem} \label{theo:Singularity} Let the notations be as above, and $n \in Q(\mu)+\Z$,  then
 $\Xi_{\Delta, r}(n , \mu, v)(z)$ is smooth at the cusps.

Moreover, let $0\ne \ell \in \Iso(V)$ be an isotropic vector and  $P_\ell$ be the associated cusp. Around the cusp ,  $$\lim_{q_{\ell}\rightarrow 0}\Xi_{\Delta, r}(n , \mu, v)(z) =0,$$
where  $q_\ell$ is a local parameter at the cusp $P_\ell$,
\end{theorem}
\begin{proof}
 Let  $D=-4N \Delta n$
and we split the proof into three cases: $D$ is not a square, $D>0$ is a square and $D=0$.

{\bf Case 1: We first assume that  $D$ is not a square.} This  case follows  directly from the  \cite[Theorem 5.1]{DY}.

{\bf Case 2: Next we assume that $D>0$ is a square.} We first work on $X_{\Gamma_\Delta}$ as $\Xi_{L^\Delta}(n, \delta, v)$ is defined over $X_{\Gamma_\Delta}$. Let $P_\ell$ be the  cusp associated to an isotropic line $\ell \in \Iso(V)$. We have, near the cusp $P_\ell$,   by  \cite[Theorem 5.1]{DY} and (\ref{eq:TWGreen})
\begin{equation} \label{eq:TWCusp}
 \Xi_{\Delta, r}(n , \mu, v)(z)=-g_{\Delta, r}(n, \mu, v, P_{\ell})\log\mid q_{\ell}\mid^2-2\psi_\Delta(n, \mu, v; q_\ell),
 \end{equation}
 where
  $\psi_\Delta(n, \mu, v; q_\ell)$ is a smooth function of $q_\ell$ (as two real  variables $q_\ell$ and $\bar{q_\ell}$) and
 $$\lim_{q_\ell\rightarrow 0} \psi_\Delta(n, \mu, v; q_\ell)=0.$$
Here
$$
g_{\Delta, r}(n, \mu, v, P_{\ell})= \frac{1}{8 \pi  \sqrt{-nv}}\beta_{3/2}(-4 \pi nv ) \alpha_{\Delta, r}(n, \mu, P_\ell),
$$

$$
\alpha_{\Delta, r}(n, \mu, P_\ell)=\sum_{w \in \Gamma_0(N)\setminus L_{r\mu }[n \Delta]} \chi_{\Delta}(w) \delta_{w, \ell}
$$
and $0\le \delta_{w,\ell} \le 2$ is the number of  isotropic lines $ \ell_w \in \Iso(V)$  which is perpendicular to $w$ and belongs to the same cusp as $\ell$. Now all terms in (\ref{eq:TWCusp}) can be descended to modular curve $X_0(N)$.

Since  $N$ is square-free,   the Atkin-Lehner involutions act on the cusps of $X_0(N)$ transitively.
It's known that $\chi_{\Delta}$ is invariant under the Atkin-Lehner involutions, so
\begin{align*}
\alpha_{\Delta, r}(n, \mu, P_\ell)
&=\sum_{w\in \Gamma_{0}(N)\setminus L_{r\mu}[n \Delta ]}\chi_{\Delta}(w)\delta_{w, \ell}
\\
&=\sum_{w\in\Gamma_{0}(N)\setminus L_{r\sigma_{\ell}\mu }[n \Delta]}\chi_{\Delta}(w)\delta_{w, \ell_{\infty}}
\\
&=\alpha_{\Delta, r}(n, \sigma_{\ell}\mu, P_{\ell_{\infty}}).
\end{align*}
Here $\sigma_{\ell} \in \SL_{2}(\Z)$ is an Atkin-Lehner operator such that $\sigma_{\ell}\ell=\ell_{\infty}$ and $L_{r\sigma_{\ell}\mu }=L+r\sigma_{\ell}\mu$.  Therefore, it suffices to show $\alpha_{\Delta, r}(n, \mu, P_{\ell_{\infty}})=0$ for any $\mu$.

Recall that $\delta_{w} =\delta_{w, \ell_{\infty}}$ is the number of  isotropic lines $ \ell\in \Iso(V)$  which is perpendicular to $w$ and belongs to the cusp $P_{\infty}$. For $w\in L_{r\mu}[n\Delta]$, $\delta_w =0$ unless $w \in L_{r\mu, \ell_\infty}[n\Delta]$ or $w \in L_{-r\mu, \ell_\infty}[n\Delta]$.

When $2r\mu \notin L$, $\delta_w=1$(see  \cite[Lemma 6.2]{DY}). From the Lemma \ref{lemmacoset}, one has
\begin{equation}\label{coefficient1}
\alpha_{\Delta, r}(n, \mu, P_{\ell_{\infty}})=\sum_{w\in \Gamma_{\ell_{\infty}}\setminus L_{r\mu, \ell_{\infty} }[n \Delta]}\chi_{\Delta}(w)~or~
  \sum_{w\in \Gamma_{\ell_{\infty}}\setminus L_{-r\mu, \ell_{\infty} }[n \Delta]}\chi_{\Delta}(w).
\end{equation}
According to Lemma \ref{charsum}, one has
\begin{equation}
\alpha_{\Delta, r}(n, \mu, P_{\ell_{\infty}})=0.
\end{equation}

When    $2r\mu \in L$, from the proof of Lemma \ref{lemmacoset},
 \begin{eqnarray}
 &&\alpha_{\Delta, r}(n, \mu, P_{\ell_{\infty}})\nonumber\\
 &=&\sum_{\delta_w=2, w \in \Gamma_{\ell_{\infty}}\setminus L_{r\mu, \ell_{\infty}}[n \Delta] }2\chi_{\Delta}(w)+\sum_{\delta_w=1, w \in \Gamma_{\ell_{\infty}}\setminus L_{r\mu, \ell_{\infty}}[n \Delta] }\chi_{\Delta}(w)\nonumber\\
 &+&\sum_{\delta_w=1, -w \in \Gamma_{\ell_{\infty}}\setminus L_{r\mu, \ell_{\infty}}[n\Delta] }\chi_{\Delta}(-w).\nonumber
 \end{eqnarray}

 Since $$\chi_{\Delta}(-w)=sgn(\Delta)\chi_{\Delta}(-w)=\chi_{\Delta}(w),$$ one has
 \begin{eqnarray}\label{coefficient2}
 \alpha_{\Delta, r}(n, \mu, P_{\ell_{\infty}})=\sum_{w \in \Gamma_{\ell_{\infty}}\setminus L_{r\mu, \ell_{\infty}}[n \Delta] }2\chi_{\Delta}(w).
 \end{eqnarray}

So from equation (\ref{coefficient2}) and Lemma \ref{charsum}, one obtains that
\begin{equation}
\alpha_{\Delta, r}(n, \mu, P_{\ell_{\infty}})=0.
\end{equation}
Then we complete the whole case $D>0$.

{\bf Case 3: We finally assume  $D=0$}. When $n=0$, we just need to consider the case $r\mu=0$.
Notice that if $r\mu\neq 0$, $L_{r\mu}[0]$ is empty.
From the  \cite[Theorem 5.1]{DY}, around cusp $P_{\ell}$
\begin{eqnarray}\label{zerocase}
&&\Xi_{\Delta, r}(0, \mu, v)\\
&=& -\sum_{\delta \in L^{\sharp}/L^{\Delta}, \delta\equiv 0(L), \atop \ell \cap L_{\delta}^\Delta\neq \phi} \chi_{\Delta}(\delta)\bigg(\frac{\varepsilon_{\ell}}{2 \pi \sqrt{vN}} (\log|q_\ell|^2)+2 \log(-\log|q_{\ell}|^2)\bigg)
    \nonumber\\
& &-2\sum_{\delta \in L^{\sharp}/L^{\Delta}, \delta\equiv 0(L), \atop \ell \cap L_{\delta}^\Delta\neq \phi}\chi_{\Delta}(\delta)\psi_\ell(0, \delta, v; q_\ell).\nonumber
\end{eqnarray}

Here $\varepsilon_{\ell}$ is the Funke constant of $\ell$ and $\psi_\ell(0, \delta, v; q_\ell)$ is smooth functions of $q_{\ell}$, and
\begin{equation}
\lim_{q_{\ell}\rightarrow 0}\psi_\ell(0, \delta, v; q_\ell)=\begin{cases} a_{\ell} &\ff \delta \in  L^\Delta\\
b_{\ell} & \ff \delta \notin L^\Delta
\end{cases}
\end{equation}
for some constant $a_{\ell}$ and $b_{\ell}$.  When $\delta \in L^\Delta$, $\chi_{\Delta}(\delta)=0$, so
\begin{equation}\label{zerolimit}
\lim_{q_{\ell}\rightarrow 0}\sum_{\delta \in L^{\sharp}/L^{\Delta}, \delta\equiv 0(L), \atop \ell \cap L_{\delta}^\Delta\neq \phi}\chi_{\Delta}(\delta)\psi_\ell(0, \delta, v; q_\ell)=\sum_{\delta \in L^{\sharp}/L^{\Delta},  \delta\equiv 0(L), \atop \ell \cap L_{\delta}^\Delta\neq \phi}\chi_{\Delta}(\delta)b_{\ell}.
\end{equation}
Combing it with equation (\ref{zerocase}), it suffices to consider $$\sum_{\delta \in L^{\sharp}/L^{\Delta}, \delta\equiv 0(L),\ell \cap L_{\delta}^\Delta\neq \phi}\chi_{\Delta}(\delta).$$

We assume that $\ell \cap L=\Z\lambda_{\ell}$, where $\lambda_{\ell}$ is the primitive element. So the representatives for
all $\delta \in L^{\sharp}/L^{\Delta}$ with $\delta\equiv r\mu\equiv 0(L)$ such that $\ell \cap L_{\delta}^\Delta\neq \phi$ are given by $$\{m\lambda_{\ell}, m=0,1,..., \Delta-1\}.$$
One has $$\sum_{\delta \in L^{\sharp}/L^{\Delta}, \delta\equiv0(L), \ell \cap L_{\delta}^\Delta\neq \phi} \chi_{\Delta}(\delta)=\sum_{m=0}^{\Delta-1} \chi_{\Delta}(m\lambda_{\ell})=0.$$
From the equation (\ref{zerocase}) and (\ref{zerolimit}), we know $\Xi_{\Delta, r}(0, \mu, v)$ is smooth around all cusps $P_{\ell}$ and goes to zero when $q_{\ell} \rightarrow 0$.
This finishes the proof of the theorem.
\end{proof}

Let
$$
Z_{\Delta, r}(n, \mu) = \begin{cases}
Z_{\Delta, r}(n, \mu) &\ff n >0,
  \\
    0 &\ff \hbox{ otherwise}.
    \end{cases}
$$

\begin{corollary}  \label{cor:GreenFunction} Let notations and assumption be as in  Theorem \ref{theo:Singularity}, then $\Xi_{\Delta, r}(n, \mu, v)$ is a Green function for $Z_{\Delta, r}(n, \mu) $   on $X_0(N)$ in the usual Gillet-Soul\'e sense, i.e.,
$$
dd^c [\Xi_{\Delta, r}(n, \mu,v) ]+ \delta_{Z_{\Delta, r}(n, \mu)} = [\omega_{\Delta, r}(n, \mu,v)].
$$

\end{corollary}

Bruinier and Ono constructed automorphic Green function for divisor $Z_{\Delta, r}(f)$ in the work \cite{BruOno}, where $f$ is the weight $1/2$ harmonic weak Maass form.  When take some special $f$, one could get the Green function $\Phi_{\Delta, r}(n, \mu)$ for divisor $Z_{\Delta, r}(n, \mu)$. Similarly as recent work of Ehlen and Sankaran \cite{ES}, we could ask a question  if $$\sum_{\mu}\sum_{n}\Phi_{\Delta, r}(n, \mu)q^ne_\mu-\sum_{\mu}\sum_{n}\Xi_{\Delta, r}(n, \mu, v)q^ne_\mu$$ is a modular form. They didn't prove case associated to modular curve in that paper.

\section{Twisted arithmetic theta function} \label{secmainresult}

In this section, we assume that $N$ is square free.

Following \cite{KM}, let $\mathcal{Y}_{0}(N)$ $ ( \mathcal{X}_{0}(N) )$ be the moduli stack over $\Z$ of cyclic isogenies of degree $N$ of elliptic
curves (generalized elliptic curves) $\pi : E\rightarrow E^{\prime}$, such that $\ker \pi$ meets every irreducible component of each geometric fiber.
The stack $\mathcal{X}_{0}(N)$ is regular and proper flat over $\Z$ and  $\mathcal{X}_{0}(N)(\C)=X_{0}(N)$. It is a DM-stack.  It is regular over $\Z$ and smooth over $\Z[\frac{1}N]$.

When $p|N$, the special fiber  $\mathcal X_0(N) \pmod p$ has two irreducible components $\mathcal X_p^\infty$ and $\mathcal X_p^0$. Let $\mathcal X_p^\infty$($\mathcal X_p^0$ ) be the component which contains the cusp $\mathcal P_\infty \pmod p$($\mathcal P_0 \pmod p$).
Here $\mathcal P_\infty$ and $\mathcal P_0$ are Zariski closure of cusp infinity and zero.
Let $\mathcal{Z}_{\Delta, r}(n, \mu)$ be the Zariski closure of $Z_{\Delta, r}(n, \mu)$.

When $\Delta=1$, $\mathcal{Z}_{\Delta, r}(n, \mu)$ has a moduli interpretation, one could check details in \cite[Section 6]{DY}.

We  define arithmetic divisor in $\widehat{\CH}^1_\R(\mathcal X_0(N))$ by
\begin{equation}
\widehat{\mathcal{Z}}_{\Delta, r}(n, \mu, v)=(\mathcal{Z}_{\Delta, r}(n, \mu), \Xi_{\Delta, r}(n , \mu, v)).
\end{equation}

The twisted arithmetic theta function ($q=e(\tau)$) is defined to be
\begin{equation} \label{eq:twistedtheta}
\widehat{\phi}_{\Delta, r}(\tau) = \sum_{\substack{ n \equiv Q(\mu) (\mod \Z)\\
                          \mu \in L^\sharp/L
                          }}
                           \widehat{\mathcal Z}_{\Delta, r}(n, \mu, v) q^n e_\mu  \in \widehat{\CH}^1_\R(\mathcal X_0(N))\otimes \C[L^\sharp/L][[q, q^{-1}]].
\end{equation}
\subsection{Arithmetic intersection}\label{sect:ArithIntersectionReview}
 The Gillet-Soul\'e intersection theory (see \cite{SouleBook}) has been extended to arithmetic divisors with log-log singularities  and equivalently metrized bundles with log singularities (\cite{BKK}, \cite{Kuhn}, \cite{Kuhn2}).  We will use K\"uhn's settings in \cite{Kuhn} as in  \cite{DY} and we refer to \cite{Kuhn} for details.

Let S=\{cusps\} and define $\widehat{\Pic}_\R(\mathcal X_{0}(N), S)$ be the group of metrized line bundles with log singularity along $S$ with $\R$-coefficients.

For an arithmetic divisor $\widehat{\mathcal Z}=(\mathcal Z, g)$ with  log-log-singularity along $S$,  $g$ is a smooth function on $X_{0}(N)\setminus \{\mathcal Z(\C)\cup S \}$, and satisfying the following conditions:
 \begin{align*}
 dd^c g& + \delta_{\mathcal Z(\C)} =[\omega],
 \\
 g(t_j) = -2 \alpha_j \log (-\log(|t_j|^2)&-2 \beta_j \log |t_j| -2 \psi_j(t_j)  \quad \hbox{near  } S_j,
 \end{align*}
 for some smooth function  $\psi_j$ and some $(1, 1)$-form $\omega$ which is smooth away from $S$. Here $t_j$ is local parameter.

We could view  the metrized line bundle $\widehat{\mathcal L}=(\mathcal L,  \| \, \|)$ as an arithmetic divisor $\widehat{\mathcal Z}=(div(s), -\log\| s\|^2)$ with canonical section $s$.

Around cusp $S_{j}$,
    $$
    \| s(t_j)\| = (- \log|t_j|^2)^{\alpha_j} |t_j|^{\ord_{S_j}(s)} \varphi(t_j).
    $$

Then one obtains that
 \begin{equation}\label{index}
 \alpha_j(g) =\alpha_j(s), \quad \beta_j(g) =\ord_{S_j}(s),  \quad \psi_j(t_j) = \log \varphi (t_j).
 \end{equation}
   Define   $\widehat{\CH}_\R^1(\mathcal X_0(N), S)$  be the quotient of the $\R$-linear combination of the
 arithmetic divisors of $\mathcal X$ with log-log growth along $S$ by $\R$-linear combinations of the principal arithmetic divisors with log-log growth along $S$.
One has $\widehat{\Pic}_\R(\mathcal X_0(N), S) \cong \widehat{\CH}_\R^1(\mathcal X_0(N), S)$.

From \cite[Proposition 1.4]{Kuhn}, there is an extended  height paring \cite[Proposition 4.1]{DY}
$$
\widehat{\CH}_\R^1(\mathcal X_0(N), S) \times \widehat{\CH}_\R^1(\mathcal X_0(N), S) \rightarrow \R,
$$
such that if $\mathcal Z_1$ and $\mathcal Z_2$ are   divisors  intersect properly, then
$$
\langle (\mathcal Z_1, g_1),  (\mathcal Z_2, g_2)  \rangle
=(\mathcal Z_1.\mathcal Z_2)_{fin} + \frac{1}2 g_1*g_2.
$$

The degree map is given by
\begin{equation} \label{eq:Degree}
\deg:  \widehat{\CH}_\R^1(\mathcal X_0(N), S)\rightarrow \R, \quad \deg (\mathcal Z,g) = \int_{X} \omega = \langle (\mathcal Z,g), (0, 2) \rangle.
\end{equation}
  It is  $\deg Z$ when $g$ is a Green function  without log-log singularity.

Define $\widehat{\CH}_\R^1(\mathcal X_0(N)):=  \widehat{\CH}_\R^1(\mathcal X_0(N), \hbox{empty})$, which is  the usual arithmetic Gillet-Soul\'e Chow group.

\subsection{The metrized Hodge bundle }
Let $\omega_N$ be the Hodge bundle on $\mathcal X_0(N)$ (see \cite{KM}). Then there is  an isomorphism $\omega_N^2 \cong \Omega_{\mathcal X_0(N)/\Z}(-S)$, which is canonically isomorphic to the line bundle of modular forms of weight $2$ for $\Gamma_0(N)$.

The normalized Petersson  metric of line bundle of modular form  $\mathcal M_k(\Gamma_0(N))$, which is given by
$$
\| f(z)\| = |f(z) (4 \pi e^{-C} y)^{\frac{k}2}|,
$$
where $k$ is the weight, $C=\frac{\log 4\pi +\gamma}{2}$ and $\gamma$ is Euler constant. We denote this metric line bundle by $\widehat{\mathcal M}_k(\Gamma_0(N))$.  This  metrized induces a metric on $\omega_N$, and we denote this metrized line bundle by $\widehat{\omega}_N$.  We will  identify
 $\widehat{\omega}_N^k\cong \widehat{\mathcal M}_k(\Gamma_0(N))$, under which
  $\widehat{\omega}_N$ becomes  the class of arithmetic divisor \begin{equation}\frac{1}k \widehat{\Div}(\Delta_N)=\frac{1}{k}(\Div\Delta_{N}, -\log \| \Delta_N(z) \|^2).
\end{equation}
Here the modular form $\Delta_N(z)$  of weight $k= 12 \varphi(N)$ is given by (\ref{eq:DeltaN}).
We need the following  two lemmas in \cite[Section 6]{DY}.

\begin{lemma}\label{divisorlemma}
\begin{equation}
\Div\Delta_{N}=\frac{rk}{12}\mathcal{P}_{\infty}-k \sum_{p|N}\frac{p}{p-1}\mathcal{X}_{p}^{0},
\end{equation}
where $$ r=N\prod_{p|N} (1+p^{-1})=[\SL_2(\Z): \Gamma_0(N)].$$
Here $\mathcal{P}_{\infty}$ is the Zariski closure of  cusp $\infty$ and $\mathcal{X}_{p}^{0}$ is a vertical component of $\mathcal{X}_{0}(N)$.
\end{lemma}

\begin{lemma} \label{lem:Delta_N}  Let   $q_z=e(z) $   be  a local parameter of $X_0(N)$ at the cusp $P_\infty$.
The metrized line bundle $\widehat{\omega}^k =\widehat{\mathcal M}_k(N)$ has log singularity along cusps with all $\alpha$-index $\alpha_P=\frac{k}2$ at every cusp $P$. At the cusp $P_\infty$, one has
    $$
    \|\Delta_N(z) \| = (-\log|q_z|^2)^{\frac{k}2}|q_z|^{\frac{r}{12} k} \varphi(q_z),
$$
with
$$
\varphi(q_z) = e^{-\frac{kC}2} \prod_{n=1}^\infty |(1- q_z)^{24 C_N(n)}|.
$$


Indeed,  $\widehat{\Div}(\Delta_N)=(\Div (\Delta_N), -\log\| \Delta_N(z) \|^2)$ is an arithmetic  divisor with log-log singularity at cusps.

\end{lemma}

\subsection{Main results}
Firstly, we prove the following proposition which is an analogue of  \cite[Proposition 6.7]{DY}.

\begin{proposition}  \label{vertical} Let $\Delta >1$, for every prime $p|N$, one has
$$
\langle \widehat{\phi}_{\Delta, r}(\tau), \mathcal X_p^0 \rangle =\langle \widehat{\phi}_{\Delta, r}(\tau), \mathcal X_p^{\infty} \rangle =0.
$$
\end{proposition}
\begin{proof}Since  $\chi_{\Delta}(-w)=\chi_{\Delta}(w),$ similarly to  work  \cite{DY}, we obtain $w_N^* (\widehat{\phi}_{\Delta, r}(\tau)) = \widehat{\phi}_{\Delta, r}(\tau)$. It is known that $ w_N^* \mathcal X_p^0= \mathcal X_p^\infty $
with $w_N=\kzxz{0}{-1}{N}{0}.$
 $w_N$ is an isomorphism, then one has
\begin{align*}
&\langle \widehat{\phi}_{\Delta, r}(\tau), \mathcal X_p^0 \rangle
 =\langle \widehat{\phi}_{\Delta, r}(\tau), \mathcal X_p^\infty \rangle
 \\
 &=\frac{1}2 \langle \widehat{\phi}_{\Delta, r}(\tau), \mathcal X_p \rangle
  =\frac{1}2 \langle  \widehat{\phi}_{\Delta, r}(\tau),(0, \log p^2)  \rangle
  \\
  &=\frac{1}2 I_{\Delta, r}(\tau, 1) \log p
  = 0 .
\end{align*}
Notice that the principal arithmetic divisor $\widehat{\Div}(p) = (\mathcal X_p, -\log p^2)$, so $\mathcal X_p=(0, \log p^2)$ in arithmetic Chow group .
\end{proof}

\begin{theorem} \label{theo:mainresult} Let the notations be as above, then

$$
\langle \widehat{\phi}_{\Delta, r}(\tau), \widehat{\omega}_N \rangle =  \frac{1}{\varphi(N)}\log(u_{\Delta})h(\Delta)\mathcal{E}_{L}(\tau, 1).
$$
\end{theorem}
\begin{proof}
We denote
\begin{equation}
\widehat{\Delta}_N = (\frac{rk}{12} \mathcal P_\infty,  -\log \| \Delta_N(z) \|^2),
\end{equation}
so
\begin{equation} \label{eq6.10}
\widehat{\Div}(\Delta_N)= \widehat{\Delta}_N  - k \sum_{p \mid N}\frac{p}{p-1} \mathcal{X}_{p}^{0}.\nonumber
\end{equation}

From the Proposition \ref{vertical}, one has
\begin{eqnarray}\langle \widehat{\phi}_{\Delta, r}(\tau), \widehat{\Div}(\Delta_N) \rangle=\langle \widehat{\phi}_{\Delta, r}(\tau), \widehat{\Delta}_N \rangle.\nonumber
\end{eqnarray}
Now it's only need to prove
\begin{equation}\langle \widehat{\phi}_{\Delta, r}(\tau), \widehat{\Delta}_N \rangle =  12 \log(u_{\Delta})h(\Delta)\mathcal{E}_{L}(\tau, 1),\end{equation}
which  amounts to check their Fourier coefficients term by term.

By Theorem \ref{theoremlift}, it is suffice to prove
\begin{equation} \label{MainIdentity}
\langle \widehat{\mathcal Z}_{\Delta, r}(n, \mu, v), \widehat{\Delta}_N \rangle   =
 -\int_{X_0(N)} \log\|\Delta_N(z)\| \omega_{\Delta, r}(n,\mu, v).
\end{equation}

From the intersection formula \cite[Proposition 1.4]{Kuhn}, we get
$$
\langle \mathcal Z_{\Delta, r}(n, \mu, v),  \widehat{\Delta}_N \rangle
=(Z_{\Delta, r}(n, \mu, v), \frac{rk}{12} \mathcal P_\infty)_{fin} + \frac{1}2 g_1*g_2.
$$
Around the cusp $P_{\infty}$, from Theorem \ref{theo:Singularity},
 $$\lim_{q_{z}\longrightarrow 0}\Xi_{\Delta, r}(n , \mu, v)(z)=0,$$
then index $\alpha_{1, j}$ and $\psi_{1, j}(0)$ in equation (\ref{index}) associated to $ \widehat{\mathcal Z}_{\Delta, r}(n, \mu, v)$ are equal to zero(see details in \cite[Proposition 4.1]{DY}). The star product here is
\begin{eqnarray}
g_1*g_2&=&\Xi_{\Delta, r}(n , \mu, v)(z)*(-\log\|\Delta_N\|^2)\nonumber\\
&=&-2\int_{X_0(N)} \log\|\Delta_N\| \omega_{\Delta, r}(n, \mu, v).\nonumber
\end{eqnarray}

It's easy to find that $\mathcal Z_{\Delta, r}(n, \mu, v)$ and $\mathcal P_\infty$ have no intersection.
Then we obtain
$$
\langle \widehat{\mathcal Z}(n, \mu, v),  \widehat{\Delta}_N \rangle
=-\int_{X_0(N)} \log\|\Delta_N\| \omega_{\Delta, r}(n, \mu, v).
$$
This finishes the proof.

\end{proof}

\section{Modularity of the arithmetic theta function}

 In this section, we will prove the modularity of $\widehat{\phi}_{\Delta, r}(\tau)$. We will follow the methods in \cite[Chapter 4]{KRYBook} and \cite[Section 8]{DY}. To simplify the notation, we denote in this section  $X=X_0(N)$ and $\mathcal X=\mathcal X_0(N)$, and let  $S$ be   the set of cusps of $X$.  Let $g_{\GS}$ be a Gillet-Soul\'e Green function for the divisor $\Div \Delta_N$(without log-log singularity), and let $\widehat{\Delta}_{\GS} =(\Div \Delta_N, g_{\GS}) \in \widehat{\CH}_\R^1(\mathcal X)$, and
$
f_N= g_{\GS} + \log\|\Delta_N\|^2$. One has
$$
\widehat{\Delta}_{\GS}= \widehat{\Div}({\Delta}_N) + a(f_N),
$$
where $ a(f_N) =(0, f_N) \in  \widehat{\CH}_\R^1(\mathcal X, S)$.

Let $A(X)$ be the space of smooth  functions $f$ on $X$ which are conjugation invariant ($Frob_\infty$-invariant), and let   $A^0(X)$ be the subspace of functions $f \in A(X)$ with
$$
\int_{X} f \mu_{GS}=0,
$$
where $\mu_{GS} =c_1(\widehat{\Delta}_{\GS})$.

 For each $p|N$, let
$\mathcal Y_p =\mathcal X_p^\infty- p \mathcal X_p^0$, and
$\mathcal Y_p^\vee =\frac{1}{\langle \mathcal Y_p, \mathcal Y_p\rangle} \mathcal Y_p$. Finally let $\widetilde{\MW}$ be the orthogonal complement of
$\R \widehat{\Delta}_{\GS}  +\sum_{p|N} \R \mathcal Y_p^\vee + \R a(1) + a(A^0(X))$ in $\widehat{\CH}_\R^1(\mathcal X)$.

 Recall the result \cite[Proposition 8.3]{DY} as follows.
 Let $$\MW=J_0(N)\otimes_\Z \R,$$ where $J_0(N)$ is the jacobian, then there is an isomorphism
\begin{equation}\label{Modellweil}
\widetilde{\MW} \cong \MW,
\widehat{\mathcal Z}=(\mathcal Z, g_Z)\mapsto  Z,
\end{equation}
where $Z$ is the generic fiber of $\mathcal Z$.

\begin{proposition} \label{prop6.2} (\cite[Propositions 4.1.2, 4.1.4]{KRYBook}, \cite[Proposition 8.2]{DY}) One has
$$
\widehat{\CH}_\R^1(\mathcal X) = \widetilde{\MW}\oplus (\R \widehat{\Delta}_{\GS}  +\sum_{p|N} \R \mathcal Y_p^\vee  + \R a(1)) \oplus a(A^0(X)).
$$
More precisely, every $\widehat{Z} =(\mathcal Z, g_Z)$ decomposes into
$$
\widehat{Z}=\widetilde{Z}_{MW} + \frac{\deg \widehat{Z}}{\deg \widehat{\Delta}_{\GS}} \widehat{\Delta}_{\GS} + \sum_{p|N} \langle \widehat{Z}, \mathcal Y_p\rangle \mathcal Y_p^\vee + 2 \kappa(\widehat{Z}) a(1) + a(f_{\widehat Z})
$$
for some $f_{\widehat Z} \in A^0(X)$, where
$$
 \kappa(\widehat{Z})\deg \widehat{\Delta}_{\GS} =\langle \widehat{Z}, \widehat{\Delta}_{\GS} \rangle
 -\frac{\deg \widehat{Z}}{\deg \widehat{\Delta}_{\GS}} \langle \widehat{\Delta}_{\GS}, \widehat{\Delta}_{\GS} \rangle.
$$
\end{proposition}

Finally, let $\Delta_z$ be the Laplacian operator with respect to $\mu_{\GS}$. Then  the space $A^0(X)$  has an orthonormal basis $\{ f_j\}$ with
$$
\Delta_z f_j + \lambda_j f_j=0,  \quad  \langle f_i, f_j \rangle = \delta_{ij}, \quad \hbox{ and } 0 < \lambda_1 < \lambda_2 < \cdots,
$$
where  the inner product is given by
$$
\langle f, g \rangle = \int_{X_0(N)} f \bar g  \mu_{\GS}.
$$
In particular, every $f \in  A^0(X)$ has the decomposition
\begin{equation} \label{eq:SpectralDecomposition}
f(z) = \sum  \langle f, f_j \rangle   f_j.
\end{equation}
Recall (\cite[(4.1.36)]{KRYBook}) that
\begin{equation}
d_z  d_z^c f =\Delta_z(f)  \mu_{\GS}.
\end{equation}

Now we could  prove the following modularity result.

\begin{theorem} Let the notation be as above. Then
$$
\widehat{\phi}_{\Delta, r}(\tau) = \tilde{\phi}_{\MW}(\tau)  + a(\phi_{SM})+\phi_{1}(\tau)a(1)
$$
where $\tilde{\phi}_{\MW}(\tau)$ is a modular form of $\Gamma'$  of weight $3/2$ and representation  $\rho_L$ valued in  finite dimension vector space $\widetilde{MW} \otimes \C[L^\sharp/L]$,  $\phi_{SM}$  and $\phi_{1}(\tau)$ are a modular forms of $\Gamma'$  of weight $3/2$ and representation  $\rho_L$ valued in $A^0(X_0(N)) \otimes \C[L^\sharp/L]$.
\end{theorem}
\begin{proof} Since $\deg(Z_{\Delta, r}(n ,\mu))=0$, one knows that the generic part $A_{\Delta, r}(\tau)\in J_0(N)\otimes \C[L^{\sharp}/L][[q]]$.
Here
\begin{equation}
A_{\Delta, r}(\tau)
  =\sum_{n >0,  \mu} Z_{\Delta, r}(n,\mu)  q_{\tau}^n e_\mu
\end{equation}
is modular by Bruinier and Ono's result \cite[Section 6]{BruOno}. Let $\widetilde{\phi}_{MW}$ be the image of $A_{\Delta, r}(\tau)$ under the isomorphism (\ref{Modellweil}), and it's also modular.

 From  decomposition result Proposition \ref{prop6.2}, one has
 \begin{equation}\label{decom}
 \widehat{\phi}_{\Delta, r}(\tau)=\widetilde{\phi}_{MW}+a( \phi_{SM})+\phi_{1}(\tau)a(1),
\end{equation}
where $\widetilde{\phi}_{MW}=(A_{\Delta, r}(\tau), g_{MW})$ and $g_{MW}$ is the  harmonic Green function and $\phi_{1}(\tau)= \frac{2}{\deg(\widehat{\Delta}_{\GS})} \langle \widehat{\phi}_{\Delta, r}, \widehat{\Delta}_{\GS} \rangle$ which is modular from the later Lemma \ref{prop7.1}. Then one has
$$
 \phi_{SM} = \Xi_{\Delta, r}(\tau, z)  - g_{MW} -\phi_{1}(\tau),
 $$
 where $ \Xi_{\Delta, r}(\tau, z)=\sum_{\mu}\sum_{n}\Xi_{\Delta, r}(n , \mu, v)q_{\tau}^{n}e_{\mu}$.
In the above decomposition, the part in $\R \widehat{\Delta}_{\GS}  +\sum_{p|N} \R \mathcal Y_p^\vee  $ is zero.

Finally, one has by  (\ref{eq:SpectralDecomposition}),
 \begin{equation}\label{spectraldecom}
 \phi_{SM}(\tau, z) =\sum_{j} \langle \phi_{SM} , f_j \rangle f_j.
 \end{equation}
 Simple calculation gives
 \begin{align*}
 \langle \phi_{SM} , f_j \rangle &=-\frac{1}{\lambda_j} \int_{X_0(N)} \phi_{SM}(\tau, z) \Delta_z(\bar{f}_j) \mu_{\GS}
 \\
 &= -\frac{1}{\lambda_j} \int_{X_0(N)} \phi_{SM}(\tau, z) d_z d_z^c \bar{f}_j
 \\
 &=-\frac{1}{\lambda_j} \int_{X_0(N)} d_z d_z^c \phi_{SM}(\tau, z)  \bar{f}_j.
 \end{align*}

 So
 \begin{eqnarray}
 -\lambda_j \langle \phi_{SM} , f_j \rangle
  =\int_{X_0(N)} d_z d_z^c (\Xi_{\Delta, r}(\tau, z) -g_{MW})f_j =
\int_{X_0(N)} \Theta_{\Delta, r}(\tau, z) f_j\nonumber
\end{eqnarray}
is modular, since $\Theta_{\Delta, r}(\tau, z)$  is modular. From the spectral decomposition (\ref{spectraldecom}), $\phi_{SM}$ is modular.

\end{proof}

From the proof of above theorem, we have the following result.

\begin{lemma}  \label{prop7.1}
$\langle \widehat{\phi}_{\Delta, r}, \widehat{\Delta}_{\GS} \rangle $ is a vector valued modular form of $\Gamma'$  valued in  $\C[L^\sharp/L]$ of weight $3/2$ and representation $\rho_L$.
\end{lemma}
\begin{proof}
 From the intersection formula \cite[Proposition 1.4]{Kuhn}, one has
\begin{eqnarray}
\langle \widehat{\phi}_{\Delta, r},  \widehat{\Delta}_{\GS}\rangle&=& \frac{1}{2} \int_{X} g_{GS} dd^{c}\Xi_{\Delta, r}(\tau, z)=\frac{1}{2} \int_{X} g_{GS} dd^{c}(\phi_{SM}+g_{MW})\nonumber
\end{eqnarray}
Since $\phi_{SM}$ and $g_{MW}$ are modular from above theorem, we know $\langle \widehat{\phi}_{\Delta, r},  \widehat{\Delta}_{\GS}\rangle$ is modular.

\end{proof}


\end{document}